\documentclass[reqno,12pt,letterpaper]{amsart}
\usepackage{amsmath,amssymb,amsthm,graphicx,mathrsfs, float, url}
\usepackage[usenames,dvipsnames]{xcolor}
\usepackage[colorlinks=true,linkcolor=Red,citecolor=Green]{hyperref}
\usepackage{tikz-cd}
\usetikzlibrary{intersections}
\usetikzlibrary{shapes,arrows,topaths,plotmarks,plothandlers,decorations,trees,folding,patterns,matrix}
\usetikzlibrary{external}
\usetikzlibrary{through}
\usetikzlibrary{plotmarks}
\usetikzlibrary{positioning}
\usetikzlibrary{calc}

\def\?[#1]{\textbf{[#1]}\marginpar{\Large{\textbf{??}}}}

\newtheorem{theo}{Theorem}
\newtheorem{prop}{Proposition}[section]
\newtheorem{defi}[prop]{Definition}

\newtheorem{lemm}[prop]{Lemma}

\numberwithin{equation}{section}

\newcommand{\RR}{\mathbb{R}}
\newcommand{\NN}{\mathbb{N}}
\newcommand{\ZZ}{\mathbb{Z}}
\newcommand{\OO}{\mathcal{O}}
\newcommand{\KE}{\mathcal{E}}
\newcommand{\KF}{\mathcal{F}}
\newcommand{\DIM}[1]{dim #1 }

\def\squarebox#1{\hbox to #1{\hfill\vbox to #1{\vfill}}}

\title[Poisson Structure on manifolds with corners]{Poisson Structure on manifolds with corners}

\author{Joel Antonio-V\'asquez}
\email{hello@joelantonio.me}

\begin{document}

\begin{abstract}
Since a Poisson Structure is a smooth bivector field, 
we use a ring-valued sheaf $\OO_{X}$ on a manifold with corners $X$,
we can interpret $\OO_{X}(U)$ as the ring of admissible smooth
functions where $U$ is an open subset on $X$, in this way, a
poisson structure on $(X, \OO_{X})$ is a sheaf morphism 
$$
\{-, -\}: \OO_{X} \times \OO_{X} \longrightarrow \OO_{X}
$$
which satisfies the Leibniz rule an also the Jacobi Identity.
\end{abstract}

\maketitle


\section{Introduction}
\label{in}

It's very know how can we describe a Poisson Structure on a smooth manifold.
Let $M$ be a smooth manifold, a poisson structure on $M$ is a smooth bivector field
$\pi$ on $M$ satisfying $[\pi, \pi]_{s} = 0$, where $[\cdot, \cdot]_{s}$ is
the Schouter bracket. Let $U$ be an open subset of $M$ an let
$F, G \in C^{\infty}(U)$ be smooth functions,
the bilinear operation $\{\cdot, \cdot\}_{U}$ is defined by 

\begin{equation}
\label{eq:b1}
\{F, G\}_{U}(m) = \langle d_{m}F \wedge d_{m}G, \pi_{m} \rangle
\end{equation}

for all $m \in M$, which satisfies the Jacobi Identity and define a Lie algebra structure on
$C^{\infty}(U)$. A lot of work already is done in classical mechanics and quantum mechanics
in \cite{lapiva}, \cite{arnold}.

Working on singular spaces is really interesting, Sorokina proposed an algorithm to describe
a poisson structure on manifolds with singularities \cite{sor}, where the spectrum of the
algebra coincides with the configuration space, it's determined by the geometry of the
singularity and described directly as the pullback.

From now, we'll focus on working on a non-smooth manifold $X$ with the following definition

\begin{defi}
\label{d:firstd}
Let $\RR^{n}_{k} = [0, \infty)^{k} \times \RR^{n-k}$. The singular space $X$ is a second countable
Hausdorff topological space in $n$ dimensions modeled on $\RR^{n}_{k}$.
\end{defi}

Let $(\NN^{k} \times \ZZ^{n-k}, U, \phi)$ and $(\NN^{l} \times \ZZ^{n-l}, V, \psi)$ be charts on $X$
where $U, V \subset X$, we notice that $\NN^{p} \times \ZZ^{n-p} \cong [0, \infty)^{p} \times \RR^{n-p}$,
where $(\NN^{k} \times \ZZ^{n-k}, U, \phi)$ is the maximal atlas of charts $(U, \phi)$ on $X$,
which in the paper of Joyce is called a manifold with corners \cite[\S 2]{joyce1}.
Joyce first presented a theory of manifolds with corners which includes a new notion of smooth map
$f: X \longrightarrow Y$ between manifolds with corners \cite{joyce1}, later he wrote a generealization
called manifolds with g-corners, extending manifolds with corners \cite{joyce2}.

The author is grateful to \textbf{Maria Sorokina} who suggested the problem consider and for plenty of helpful discussions.

\section{Definitions}

We recall Definition \ref{d:firstd}, a $n$-dimensional manifold with corners is a paracompact Hausdorff Topological space $X$ equipped with a maximal
$n$-dimensional atlas with corners.

\begin{defi}
\label{d:secondd}
Let $X$ be a $n$-manifold with corners and $(U, \phi)$ a chart on the manifold $X$ where $U \subset X$, then $f: X \longrightarrow \RR$ is a smooth map. We can define a smooth map between subsets of $\RR^{n}, \RR$ in the sense of \cite[\textsection 2.1]{joyce1} such that $f \circ \phi: U \longrightarrow \RR$.  We write $C^{\infty}(X)$ for the $\RR$-algebra of smooth function $f:X \longrightarrow \RR$.
\end{defi}

\begin{defi}
\label{d:thirdd}
Let $X, Y$ and $Z$ be manifolds with corners and let $f: X \longrightarrow Z$, $g: Y \longrightarrow Z$ be smooth maps, we define a fibre product $W = X \times_{f, Y, g} Z$ of manifolds with corners, such the dimension of $W$ is $n$, where $n = \DIM{X} + \DIM{Y} - \DIM{Z}$. 
\end{defi}

\begin{lemm}
  \label{l:firstl}
We can describe $W = \{ (x, y) \in X \times Y : f(x) = g(y) \}$ such that $W \subseteq X \times Y$ because Definition \ref{d:thirdd}. Let $\pi_{x}: W \longrightarrow X$ and
$\pi_{y}: W \longrightarrow Y$ be map proyections, let us consider $(U, \phi)$ be the maximal atlas on $W$, where $W \subseteq \RR^{n}_{k}$. If $\phi: U \longrightarrow W$, then $\pi_{x} \circ \phi: U \longrightarrow X$ and $\pi_{y} \circ \phi: U \longrightarrow Y$ are smooths maps for all $u \in U$ with $\phi(u) = (x, y)$, in this way

\begin{equation}
  \label{secondeq}
  d(\pi_{x} \circ \phi) |_{u} \oplus d(\pi_{y} \circ \phi) |_{u}: T_{u}U = \RR^{n} \longrightarrow T_{x}X \oplus T_{y}Y
\end{equation}

\begin{figure}[H]
  \centering
  \begin{tikzpicture}[baseline= (a).base]
    \node[scale=1.3] (a) at (0, 0){
    \begin{tikzcd}
    U\arrow[leftarrow]{drr}{\phi \ \circ \ \pi_{y}}
    \arrow[rightarrow]{dr}[description]{\phi}
    \arrow[dotted]{ddr}[description]{\phi \ \circ \ \pi_{x}} & & \\
                                                             & W \arrow{r}[swap]{\pi_{y}} \arrow{d}{\pi_{x}} & Y \arrow{d}[swap]{g} \\
                                           & X \arrow{r}{f} & Z
  \end{tikzcd}
  };
  \end{tikzpicture}
\end{figure}

Since $W = X \times_{f, Z, g} Y$, then we have a pushout square of $C^{\infty}$-rings.
\end{lemm}

\begin{figure}[H]
  \centering
  \begin{tikzpicture}[baseline= (a).base]
  \node[scale=1.3] (a) at (0, 0){
      \begin{tikzcd}
        C^{\infty}(Z) \arrow{r}{g^{*}} \arrow{d}{f^{*}}
        & C^{\infty}(Y) \arrow{d}{\pi_{y}^{*}}\\
        C^{\infty}(X) \arrow{r}{\pi_{x}^{*}} & C^{\infty}(W)
      \end{tikzcd}
    };
  \end{tikzpicture}
\end{figure}

Such that $C^{\infty}(W) = C^{\infty}(X) {\scriptstyle\coprod}_{f^{*}, C^{\infty}(Z), g^{*}} C^{\infty}(Y)$.

\begin{defi}
  \label{d:fourthd}
Let $X, Y$ be manifolds with corners and let $Z$ a manifold without boundary in the sense of \cite[\textsection 2]{joyce1}, if $f: X \longrightarrow Z$ and $g: Y \longrightarrow Z$ then $f \circ i_{X}: \partial X \longrightarrow Z$ and $g: Y \longrightarrow Z$ are transverse
\end{defi}

\begin{equation}
  \label{thirdeq}
  \partial(X \times_{f, Z, g} Y) \cong (\partial X \times_{f \circ i_{X}, Z, g} Y) {\scriptstyle\coprod} (X \times_{f, Z, g \circ i_{Y}} \partial Y).
\end{equation}

\begin{defi} (\textbf{Presheaf})
\label{d:fifthd}
Let $X$ be a manifold with corners. A presheaf of an abelian group on $X$ consists of two sets of data:

\begin{itemize}
  \item{
      \textbf{Sections over open sets}, for each open set $U \subseteq X$ an abelian group $\Gamma(U)$, also written $\Gamma(U, F)$. The elements of $F(U)$ are called \textbf{sections} of $F$ over $U$.
    }
  \item{\textbf{Restriction maps}, for every inclusion $V \subseteq U$ of open sets in $X$ a group homomorphism $p^{U}_{V}: F(U) \longrightarrow F(V)$ subjected to the conditions 

      $$
        p^{U}_{W} = p^{V}_{W} \circ p^{U}_{V}
      $$ 

      for all sequences $W \subseteq V \subseteq U$ of inclusions of open sets in $X$. The maps $p^{U}_{V}$ are called \textbf{restriction maps} and if $s$ is  a section over $U$, the restriction $p^{U}_{V}(s)$ is often written as $s|_{V}$.
    }
\end{itemize}
\end{defi}

\begin{lemm}
  \label{l:secondl}
Suppose that $\KE, \KF$ are presheaves of abelian groups in $X$. A morphism $\phi: \KE \longrightarrow \KF$ such that $\phi(U): \KE(U) \longrightarrow \KF(U)$ for all open $U \subseteq X$, such that following diagram commutes for all open $V \subseteq U \subseteq X$

\begin{figure}[H]
  \centering
  \begin{tikzpicture}[baseline= (a).base]
  \node[scale=1.3] (a) at (0, 0){
      \begin{tikzcd}
        \KE(U) \arrow{r}{\phi(U)} \arrow{d}{p_{UV}}
        & \KF(U) \arrow{d}{p'_{UV}}\\
        \KE(U) \arrow{r}{\phi(V)} & \KF(U)
      \end{tikzcd}
    };
  \end{tikzpicture}
\end{figure}

where $p_{UV}$ is the restriction map for $\KE$, and $p'_{UV}$ is the restriction map for $\KF$.
\end{lemm}

\begin{defi} (\textbf{Stalk})
Let $\KE$ be a presheaf of abelian groups on $X$ The Stalk $\KE_{x}$ is the direct limit of the group $\KE(U)$ for all $x \in U \subseteq X$, via the restrictions
maps $p_{UV}$.
\end{defi}

\begin{defi} (\textbf{Sheaves})
A sheaf of abelian groups $F$ on $X$ is a presheaf of abelian groups on $X$ satisfying the two following requirements:

\begin{itemize}
  \item{
      \textbf{Locally axiom}, let $\{U_{i}\}_{i \in I}$ be an open cover of the open set $U$ and let $s$ be a section of $F$ over $U$. If the restrictions of $s$ to $U_{i}$ all vanish (i.e., one has $s|_{U_{i}}=0$ for all $i$, then $s=0$).
    }
  \item{
      \textbf{Gluing axiom}, let $\{U_{i}\}_{i \in I}$ be an open cover of the open set $U$. Given sections $s_{i}$ over $U_{i}$ matching on the intersections $U_{ij} = U_{i} \cap U_{j}$, (i.e., $s_{i}|_{U_{i} \cap U_{j}} = s_{j}|_{U_{i} \cap U_{j}}$), then there is a section $s$ of $F$ over $U$ satisfying $s|_{U_{i}} = s_{i}$.
    }
\end{itemize}
\end{defi}

\begin{figure}[H]
  \begin{tikzpicture}
    \node (a) at (0,0) {$0$};
    \node (b) at (2,0) {$F(U)$};
    \node (c) at (4.5,0) {$\prod_{i}F(U_{i})$};
    \node (d) at (8,0) {$\prod_{ij}F(U_{i} \cap U_{j})$};
    \path[->,font=\scriptsize,>=angle 90]
    (a) edge node[above] {} (b);
    \path[->,font=\scriptsize,>=angle 90]
    (b) edge node[above] {$\alpha$} (c);
    \path[->,font=\scriptsize,>=angle 90]
    (c) edge node[above] {$p$} (d);
  \end{tikzpicture}
\end{figure}

Let $\phi : \KE \longrightarrow \KF$ a morphism, which induces $\phi_{x} : \KE_{x} \longrightarrow \KF_{x}$ for all $x \in X$. If $\KE, \KF$ are sheaves then $\phi$ is an isomoprhism if and only if $\phi_{x}$ is an isomorphism for all $x \in X$. The reader can read more about sheaves in the sense of manifold with corners in \cite{huma}, \cite[\textsection 4]{joyce3}, \cite[\textsection 3]{joyce4}.

\section{Poisson Structure}
Since manifolds with corners allows us to define a structure sheaf on it, we can describe a poisson structure in the following way

\begin{theo}
\label{prop3.1}
A poisson structure on $(X, \OO)$, where $X$ is a manifold with corners, is a sheaf morphism
\begin{equation}
  \label{fith_eq}
  \{-, -\} = \OO \times \OO \longrightarrow \OO
\end{equation}
that is a derivation (satisfies the Leibniz rule) in each argument and also satisfies the Jacobi Identity.
\end{theo}
\begin{proof}
  Let $X$ be a manifold with corners and $\OO_{x}$ a ring-valued sheaf (structure sheaf) on $X$, then $\OO_{x}(U)$ can be interpreted as the ring of admissible smooth functions on an open subset $U \subset X$ (i.e., $C^{\infty}$-ring $\OO_{x}(U)$) by Lemma \ref{l:firstl}, for each open sets $V \subseteq U \subseteq X$ we are given a morphism of $C^{\infty}$-rings $p_{UV} : \OO_{x}(U) \longrightarrow \OO_{x}(V)$, which by Definition \ref{d:fifthd} is the restriction map by Lemma \ref{l:secondl}, then $\OO_{x}$ is a sheaf of $C^{\infty}$-rings on $X$.

  Let $X$ and $Y$ be manifolds with corners and $F: X \longrightarrow Y$ a weakly smooth map in the sense of \cite[\textsection 3]{joyce1}. Define $(X, \OO_{X})$ and $(Y, \OO_{Y})$, for all open $U \subseteq Y$ we define a morphism of sheaves of $C^{\infty}$-rings on $Y$ in the following way 

  $$
  f_{\#}(U) : \OO_{Y}(U) = C^{\infty}(U) \rightarrow \OO_{X}(f^{-1}(U)) = C^{\infty}(f^{-1}(U))
  $$

  by $f_{\#}(U) : c \mapsto c \circ f$ for all $c \in C^{\infty}(U)$, and $f_{\#} : \OO_{Y} \longrightarrow f_{*}(\OO_{X})$ is a morphism of sheaves of $C^{\infty}$-rings on $Y$.
\end{proof}

Let $f, g \in \OO_{X}$ be smooth functions and $s$ is a local section of $F$, where $F$ is the bracket $\{-, -\} : \OO_{X} \times \OO_{X} \longrightarrow \OO_{X}$ , then by Theorem \hyperref[prop3.1]{3.1} the bracket is a derivation in the first argument and clearly satisfies the Leibniz identity
$$
\{f, gs\} = \{f, g\}s + g . \{f, s\}
$$

\section{Example}
Let $X$ be a manifold with corners, we define a $C^{\infty}$-ringed space $\underline{X} = (X, \OO_{X})$ and $\OO_{X}(U) = C^{\infty}(U)$ for each $U \subseteq X$, where $C^{\infty}(U)$ is the $C^{\infty}$-ring of smooth maps $f: U \longrightarrow \RR$, and if $V \subseteq U \subseteq X$ are open we define $p_{UV}: C^{\infty}(U) \rightarrow C^{\infty}(V)$ by $p_{UV}: f \mapsto f|_{V}$.

For each $x \in X$, the Stalk $\OO_{X, x}$ is the $C^{\infty}$-local ring of germs $[(f, U)]$ of smooth functions $f: X \longrightarrow \RR$ at $x \in X$, with a unique maximal ideal $\mathfrak{m}_{X, x}: \{[(f, U)] \in \OO_{X, x}: f(x) = 0$ \} amd $\OO_{X, x}/\mathfrak{m}_{X, x} \cong \RR$. Hence $\underline{X}$ is a local $C^{\infty}$-ringed space, so we can define a poisson structure in the sense of Theorem \hyperref[prop3.1]{3.1}.

\section{Conclusions}
A $C^{\infty}$-ringed space $(X, \OO_{X})$ allows us to define smooth functions where we can define a poisson structure (which is a sheaf morphism) on a manifold with corners in $n$ dimensions. Thanks to the definition of manifolds with cornes given by Joyce, we can adopt aplications in symplectic theory, where we can define a poisson structure without problems.	


\def\arXiv#1{\href{http://arxiv.org/abs/#1}{arXiv:#1}}

\end{document}